\documentclass[11pt,paper]{amsart}

\usepackage{amssymb,latexsym}
\usepackage[matrix,arrow]{xy}
\usepackage[usenames]{color}
\usepackage{slashed}
\usepackage{parskip}
 \numberwithin{equation}{section}
\newtheorem{theorem}{Theorem}[section]

\newtheorem{lemma}{Lemma}[section]

\theoremstyle{definition}

\theoremstyle{remark}
\newtheorem{remark}{Remark}[section]

\def\proof{{\indent\it Proof.}}
\def\endproof{\hfill\hbox{$\sqcup$}\llap{\hbox{$\sqcap$}}\medskip}
\begin{document}
\title[Bounds for eigenvalue ratios]
{Bounds for eigenvalue ratios of the Laplacian*}
\author [Q. -M. Cheng and X. Qi]{ Qing-Ming Cheng  and Xuerong Qi}
\address{Qing-Ming Cheng\\  Department of Mathematics, Graduate School of Science and Engineering,
Saga University, Saga 840-8502,  Japan, cheng@cc.saga-u.ac.jp}
\address{Xuerong Qi\\  Department of Mathematics,  Graduate School  of Science and Engineering,
Saga University, Saga 840-8502,  Japan, qixuerong609@gmail.com}
\subjclass{}
\renewcommand{\thefootnote}{\fnsymbol{footnote}}

\footnotetext{2010 \textit{Mathematics Subject Classification}:
35P15, 58C40}

\footnotetext{* Research partially supported by a Grant-in-Aid for
Scientific Research from JSPS.}

\footnotetext{{\it Key words and phrases}:    Laplacian, eigenvalue,
Dirichlet eigenvalue problem}

\begin{abstract}

For a bounded domain $\Omega$ with a piecewise smooth boundary  in
an $n$-dimensional Euclidean space $\mathbf{R}^{n}$, we study
eigenvalues of the Dirichlet eigenvalue problem of the Laplacian.
First we give a general inequality for eigenvalues of the Laplacian.
As an application, we study lower order eigenvalues of the Laplacian
and derive the ratios of lower order eigenvalues of the Laplacian.
\end{abstract}
\maketitle
\renewcommand{\sectionmark}[1]{}
%-------------------------------------------------------------------------

\section{introduction}

 Let $\Omega \subset \mathbf{R}^{n}$ be a bounded domain with a piecewise smooth boundary
 $\partial \Omega$    in an
 $n$-dimensional Euclidean space $\mathbf{R}^{n}$. We consider the
 following Dirichlet eigenvalue problem of the Laplacian:
\begin{equation}
\left \{ \aligned \Delta u&=-\lambda u \quad \text{in \ \ $\Omega$}, \\
u&=0 \quad \ \ \ \ \text{on $\partial \Omega$}. \endaligned \right.
\end{equation}
 It is well known that the spectrum of this problem is   real
 and  discrete:
 $$
 0<\lambda_1< \lambda_2\leq \lambda_3\leq \cdots\nearrow\infty,
$$
where each $\lambda_i$ has finite multiplicity which is repeated
according to its multiplicity.

The investigation of universal bounds for eigenvalues of the
eigenvalue problem (1.1)
 was initiated by Payne, P\'olya and
Weinberger \cite{PPW1}. In 1956, they showed that for
$\Omega\subset\mathbf{R}^{2}$, the ratio of the first two
eigenvalues satisfies
\begin{equation}
\frac{\lambda_2}{\lambda_1}\leq3;
\end{equation}
they further conjectured that one could replace the value 3 here by
the value that $\frac{\lambda_2}{\lambda_1}$ assumes when $\Omega$
is a disk, approximately 2.539. With respect to the conjecture of
Payne, P\'olya and Weinberger, many mathematicians studied it. For
examples, Brands \cite{B}, de Vries \cite{dV}, Chiti \cite{Chi},
Hile and Protter \cite{HP}, Marcellini \cite{Ma} and so on. Finally
Ashbaugh and Benguria \cite{AB2} (cf. \cite{AB1} and \cite{AB3})
solved this conjecture.

 For $\Omega\subset\mathbf{R}^{2}$, Payne, P\'olya and
Weinberger \cite{PPW1} also showed that
\begin{equation}
\frac{\lambda_2+\lambda_3}{\lambda_1}\leq 6.
\end{equation}
Subsequent to the paper of Payne, P\'olya and Weinberger, many
mathematicians improved the constant 6 in (1.3). Specifically, in
1964, Brands \cite{B} obtained $3+\sqrt7$; then in 1980, Hile and
Protter \cite{HP} obtained 5.622; Marcellini \cite{Ma} obtained
$\frac{15+\sqrt{345}}{6}$; and
 very recently, Chen and Zheng \cite{CZ} have obtained 5.3507.
 Furthermore, under the condition
$\frac{\lambda_2}{\lambda_1}\geq2-\frac{\lambda_1}{\lambda_4},$ Chen
and Zheng have also proved that
\begin{equation}
\frac{\lambda_2+\lambda_3}{\lambda_1}\leq
5+\frac{\lambda_1}{\lambda_4}.
\end{equation}

 For a
general dimension $n\geq 2$, Ashbaugh and Benguria \cite{AB4} (cf.
Thompson \cite{T} ) proved
 \begin{equation}
 \frac{\lambda_2+\lambda_3+\cdots +\lambda_{n+1}}{\lambda_1}\leq
 n +4.
 \end{equation}
 Furthermore,  Ashbaugh and Benguria \cite{AB4} (cf. Hile and Protter \cite{HP} )
improved the result (1.5) to
 \begin{equation}
 \frac{\lambda_2+\lambda_3+\cdots +\lambda_{n+1}}{\lambda_1}\leq
 n+3 +\dfrac{\lambda_1}{\lambda_2}.
 \end{equation}
In this paper, by making use of the fact that eigenfunctions form an
orthonormal  basis of $L^2(\Omega)$ in place of the Rayleigh-Ritz
formula, we obtain a general inequality for eigenvalues of the
Laplacian. As an application,  we study lower order eigenvalues of
the Laplacian and obtain the following:

\begin{theorem}  Let $\Omega$ be a bounded domain with a piecewise smooth boundary $\partial\Omega$ in an
$n$-dimensional Euclidean space $\mathbf{R}^{n}$. Then, for $ 1\leq
i\leq n+2$, eigenvalues of the eigenvalue problem {\rm (1.1)}
satisfy at least one of the following:
\begin{enumerate}
\item
$\dfrac{\lambda_2}{\lambda_1}< 2-\dfrac{\lambda_1}{\lambda_i},$
\item
$\dfrac{\lambda_2+\lambda_3+\cdots
+\lambda_{n+1}}{\lambda_1}\leq
 n+3
 +\dfrac{\lambda_1}{\lambda_i}.$
   \end{enumerate}
\end{theorem}

\begin{remark} Taking $i=1$ in the theorem 1.1, we derive the result
(1.5) of Ashbaugh and Benguria. Taking $i=2$ in the theorem 1.1, we
get the result (1.6) of Ashbaugh and Benguria. Taking $n=2, \ i=4$,
we have  the result (1.4) of Chen and Zheng \cite{CZ}.
\end{remark}

%-------------------------------------------------------------------------

\section{Main results and proofs}

Let $u_j$ be the
orthonormal eigenfunction corresponding to the $j$-th eigenvalue
$\lambda_j$, i.e. $u_j$ satisfies
\begin{equation}
\left \{ \aligned &\Delta u_j=-\lambda_j u_j \quad \text{in \ \ $\Omega$}, \\
&u_j=0 \quad\quad\quad \ \ \ \text{on $\partial \Omega$},\\
&\int_\Omega u_ju_k=\delta_{jk}. \endaligned \right.
\end{equation}
 In this section, first of all, by making use of the fact that $\{u_j\}_{j=1}^{\infty}$ forms an
orthonormal  basis of $L^2(\Omega)$ in place of the Rayleigh-Ritz
formula, we obtain a general inequality for eigenvalues of the
Laplacian.

\begin{theorem} Let $\Omega$ be a bounded domain with a piecewise smooth boundary $\partial\Omega$ in an
$n$-dimensional Euclidean space $\mathbf{R}^{n}$. Then, there exists
a Cartesian coordinate system $(z_1,\cdots,z_n)$ of
$\mathbf{R}^{n}$, such that,
 eigenvalues of the eigenvalue problem {\rm (1.1)} satisfy
\begin{equation}\sum_{\alpha=1}^{n}\frac{\lambda_{k+1}-\lambda_1}
{1+\sum\limits_{j=\alpha+1}^{k}(\lambda_{k+1}-\lambda_j)a_{\alpha
j}^{2}}\leq 3\lambda_1+\frac{\lambda_1^{2}}{\sigma_l},\end{equation}
where
$$
a_{\alpha j}=\int_{\Omega}z_{\alpha}u_1u_{j},\quad
\sigma_l=\lambda_1+\dfrac{\lambda_{l+1}-\lambda_1}
{1+\sum\limits_{j=2}^{l}\dfrac{\lambda_{l+1}-\lambda_j}{\lambda_{j}-\lambda_1}
\biggl[1-(\lambda_{j}-\lambda_1)\sum\limits_{\alpha=1}^{j-1}a_{\alpha
j}^2\biggl]}.
$$
\end{theorem}

\vskip 2mm
\noindent {\it Proof of Theorem 2.1.}   Let
$x_1,\cdots,x_n$ be standard coordinate functions in
 $\mathbf{R}^{n}$.  We consider  the $n\times n$-matrix $A=(A_{\alpha\beta})$ defined by
 $$A_{\alpha\beta}=\int_{\Omega} x_{\alpha} u_1u_{\beta+1}.
$$ From the orthogonalization of Gram and Schmidt, there exist an
upper triangle matrix $R=(R_{\alpha \beta})$ and an orthogonal
matrix $Q=(q_{\alpha \beta})$ such that $R=QA$. Thus,
$$
R_{\alpha \beta}=\sum_{\gamma=1}^nq_{\alpha\gamma}A_{\gamma\beta}
=\int_{\Omega} \sum_{\gamma=1}^nq_{\alpha\gamma}x_{\gamma}
u_1u_{\beta+1}=0, \ \text{for} \ 1\leq \beta <\alpha\leq n.
$$
Defining
$y_{\alpha}=\sum\limits_{\gamma=1}^nq_{\alpha\gamma}x_{\gamma}$, we
have
$$
\int_{\Omega}  y_{\alpha} u_1u_{\beta+1}=\int_{\Omega}
\sum_{\gamma=1}^nq_{\alpha\gamma}x_{\gamma} u_1u_{\beta+1}=0, \
\text{for} \ 1\leq \beta <\alpha\leq n.
$$
Putting
$$
z_{\alpha}= y_{\alpha}-y_{\alpha}^{(0)}, \qquad
y_{\alpha}^{(0)}=\int_{\Omega} y_{\alpha}u_1^2, \quad  \text{for} \
\ 1\leq \alpha \leq n
$$
and
$$
a_{\alpha j}=\int_{\Omega}z_{\alpha}u_1u_{j},
$$
we have
\begin{equation}
a_{\alpha j}=0,\qquad \mbox{for}\ 1\leq j\leq \alpha \leq n.
\end{equation}
Defining
$$
b_{\alpha j}=\int_{\Omega}u_j \nabla z_\alpha \cdot \nabla u_1, \ \
$$
from integration by parts, we obtain
\begin{equation*}
\begin{aligned}
-\lambda_ja_{\alpha j}&=\int_{\Omega}z_\alpha u_1\Delta u_j
=\int_{\Omega}\Delta(z_\alpha u_1)u_j\\
&=\int_{\Omega}\biggl( 2\nabla z_\alpha \cdot\nabla u_1-
\lambda_1 z_\alpha u_1 \biggl)u_j=2b_{\alpha j}-\lambda_1a_{\alpha j},
\end{aligned}
\end{equation*}
namely,
\begin{equation}
2b_{\alpha j}=(\lambda_1-\lambda_j)a_{\alpha j}.
\end{equation}

Since $\{u_j\}_{j=1}^{\infty}$ is an orthonormal basis in $L^2(\Omega)$,  we have
\begin{equation}
z_\alpha u_1=\sum_{j=\alpha+1}^{\infty}a_{\alpha j}u_j \ \ {\rm and}
\ \
  \|z_\alpha u_1\|^2=\sum_{j=\alpha+1}^{\infty}a_{\alpha j}^2.
\end{equation}
Furthermore,
\begin{equation}
2\int_{\Omega}z_\alpha u_1 \nabla z_\alpha\cdot \nabla u_1
=2\sum_{j=\alpha+1}^{\infty}a_{\alpha j}b_{\alpha j}
=\sum_{j=\alpha+1}^{\infty}(\lambda_1-\lambda_j)a_{\alpha j}^2.
\end{equation}
On the other hand, from integration by parts, we get
\begin{equation*}
 -2\int_{\Omega}z_\alpha u_1 \nabla  z_\alpha \cdot \nabla
u_1=-\frac{1}{2}\int_{\Omega} \nabla z_\alpha^{2} \cdot \nabla
u_1^{2}=\frac{1}{2}\int_{\Omega}u_1^{2} \Delta z_\alpha^{2}=1.
\end{equation*}
Hence we have
\begin{equation}
\begin{aligned}
\sum_{j=\alpha+1}^{\infty}(\lambda_j-\lambda_1)a_{\alpha j}^2=1.
\end{aligned}
\end{equation}
For any positive integer $k$,  we obtain
\begin{equation*}
\begin{aligned}
\sum_{j=\alpha+1}^{\infty}(\lambda_j-\lambda_1)a_{\alpha j}^2
=&\sum_{j=\alpha+1}^{k}(\lambda_j-\lambda_1)a_{\alpha j}^2+\sum_{j=k+1}^{\infty}(\lambda_j-\lambda_1)a_{\alpha j}^2\\
\geq &\sum_{j=\alpha+1}^{k}(\lambda_j-\lambda_1)a_{\alpha j}^2
+(\lambda_{k+1}-\lambda_1)\sum_{j=k+1}^{\infty}a_{\alpha j}^2\\
=&\sum_{j=\alpha+1}^{k}(\lambda_j-\lambda_1)a_{\alpha j}^2
+(\lambda_{k+1}-\lambda_1)\sum_{j=\alpha+1}^{\infty}a_{\alpha j}^2-(\lambda_{k+1}-\lambda_1)\sum_{j=\alpha+1}^{k}a_{\alpha j}^2\\
=&\sum_{j=\alpha+1}^{k}(\lambda_j-\lambda_{k+1})a_{\alpha j}^2
+(\lambda_{k+1}-\lambda_1)\sum_{j=\alpha+1}^{\infty}a_{\alpha j}^2.
\end{aligned}
\end{equation*}
Thus, we infer
\begin{equation}
\begin{aligned}
(\lambda_{k+1}-\lambda_1) \|z_\alpha u_1\|^2\leq 1+
\sum_{j=\alpha+1}^{k}(\lambda_{k+1}-\lambda_j)a_{\alpha j}^2.
\end{aligned}
\end{equation}

 For some real constant
 $t>\frac{1}{2}$, from
integration by parts, we get
 \begin{equation}\aligned
\int_\Omega |\nabla u_{1}^{t-1}|^{2}u_1^{2}&=(t-1)^{2}\int_\Omega
u_1^{2t-2}|\nabla u_{1}|^{2}\\
&=\frac{(t-1)^{2}}{2t-1}\int_\Omega
\nabla u_{1}\cdot \nabla u_1^{2t-1}\\
&=-\frac{(t-1)^{2}}{2t-1}\int_\Omega
 u_1^{2t-1}\Delta u_{1}\\
&=\frac{(t-1)^{2}}{2t-1}\lambda_1\int_\Omega
 u_1^{2t}.
\endaligned
\end{equation}
 Letting
$$
d_j=\int_{\Omega}u_1^{t}u_j,
$$
we know
\begin{equation}
\begin{aligned}
u^t_1=\sum_{j=1}^{\infty}d_{j}u_j, \qquad
\|u^t_1\|^2=\int_{\Omega}u_1^{2t}=\sum_{j=1}^{\infty}d_{j}^2.
\end{aligned}
\end{equation}
Since for any function $f\in C^2(\Omega)\cap C(\bar\Omega)$, from
integration by parts, we have
\begin{equation}
-2\int_{\Omega}fu_1 \nabla  f \cdot \nabla u_1
=\int_{\Omega}u_1^2f\Delta f+\int_{\Omega}|\nabla  f|^2u_1^2.
\end{equation}
Taking $f=u_1^{t-1}$ in (2.11), we get
\begin{equation*} \aligned \int_\Omega |\nabla
u_{1}^{t-1}|^{2}u_1^{2}&=-2\int_{\Omega}u_1^{t} \nabla  u_1^{t-1}
\cdot \nabla u_1 -\int_{\Omega}u_1^{t+1}\Delta
u_1^{t-1}\\
&=-\sum_{j=1}^{\infty}d_j\biggl(2\int_{\Omega}u_j \nabla  u_1^{t-1}
\cdot \nabla u_1+\int_{\Omega}u_ju_1\Delta
u_1^{t-1}\biggl)\\
&=-\sum_{j=1}^{\infty}d_j\biggl(\int_{\Omega}u_j \Delta
u_1^{t}-\int_{\Omega}u_ju_1^{t-1}\Delta
u_1\biggl)\\
&=-\sum_{j=1}^{\infty}d_j\biggl(\int_{\Omega} u_1^{t} \Delta
u_j-\int_{\Omega}u_ju_1^{t-1}\Delta
u_1\biggl)\\
&=\sum_{j=1}^{\infty}d_j\biggl(\lambda_j\int_{\Omega} u_1^{t}
u_j-\lambda_1\int_{\Omega}u_ju_1^{t}\biggl)\\
&=\sum_{j=2}^{\infty}(\lambda_j-\lambda_1)d_j^{2}.\\
\endaligned
\end{equation*}

Thus, by (2.9), we get
\begin{equation}
\int_\Omega |\nabla
u_{1}^{t-1}|^{2}u_1^{2}=\frac{(t-1)^{2}}{2t-1}\lambda_1\int_\Omega
 u_1^{2t}=\sum_{j=2}^{\infty}(\lambda_j-\lambda_1)d_j^{2}.
\end{equation}
Taking
\begin{equation}\beta_j=\frac{d_j}{\sqrt{\frac{(t-1)^{2}}{2t-1}\lambda_1\int_\Omega
 u_1^{2t}}},\end{equation}
 then,  we have
  \begin{equation}
\sum_{j=2}^{\infty}(\lambda_j-\lambda_1) \beta_{j}^{2}=1.
\end{equation}
For any positive integer $l$,  we obtain
\begin{equation*}
\begin{aligned}
\sum_{j=2}^{\infty}(\lambda_j-\lambda_1)\beta_{j}^2
=&\sum_{j=2}^{l}(\lambda_j-\lambda_1)\beta_{j}^2+\sum_{j=l+1}^{\infty}(\lambda_j-\lambda_1)\beta_{j}^2\\
\geq &\sum_{j=2}^{l}(\lambda_j-\lambda_1)\beta_{j}^2
+(\lambda_{l+1}-\lambda_1)\sum_{j=l+1}^{\infty}\beta_{j}^2\\
=&\sum_{j=2}^{l}(\lambda_j-\lambda_1)\beta_{j}^2
+(\lambda_{l+1}-\lambda_1)\sum_{j=2}^{\infty}\beta_{j}^2-(\lambda_{l+1}-\lambda_1)\sum_{j=2}^{l}\beta_{j}^2\\
=&\sum_{j=2}^{l}(\lambda_j-\lambda_{l+1})\beta_{j}^2
+(\lambda_{l+1}-\lambda_1)\sum_{j=2}^{\infty}\beta_{j}^2.
\end{aligned}
\end{equation*}
Thus, we infer
\begin{equation}
\begin{aligned}
(\lambda_{l+1}-\lambda_1) \sum_{j=2}^{\infty}\beta_{j}^2\leq 1+
\sum_{j=2}^{l}(\lambda_{l+1}-\lambda_j)\beta_{j}^2.
\end{aligned}
\end{equation}

Set
 $$B(t)=\frac{\int_\Omega
 u_1^{2t}}{\left(\int_\Omega
 u_1^{t+1}\right)^{2}}.$$
From (2.10), (2.12), (2.13) and (2.14), we infer
\begin{equation}
\begin{aligned}
\dfrac{(t-1)^2}{2t-1}\lambda_1\sum_{j=1}^{\infty}\beta_{ j}^2=1.
\end{aligned}
\end{equation}
Since
\begin{equation*}
\begin{aligned}&\biggl(\int_{\Omega}u_1^{t +1}\biggl)^2
=d_{1}^2=\dfrac{(t-1)^2}{2t-1}\lambda_1\beta_1^2\int_{\Omega}u_1^{2t},
\end{aligned}
\end{equation*}
according to the definition of $B(t)$, we have
\begin{equation}
\begin{aligned}&
B(t)=\dfrac{\int_{\Omega}u_1^{2t}}{\biggl(\int_{\Omega}u_1^{t
+1}\biggl)^2} =\dfrac1{\dfrac{(t-1)^2}{2t-1}\lambda_1\beta_1^2}=
\dfrac1{1-\dfrac{(t-1)^2}{2t-1}\lambda_1\sum\limits_{j=2}^{\infty}\beta_{
j}^2}.
\end{aligned}
\end{equation}
By integration by parts, we have
$$
d_1=\int_{\Omega}u_1^{t+1}=\frac{1}{2}\int_{\Omega} u_1^{t+1}\Delta z_\alpha^{2}
=-(t+1)\int_{\Omega}z_\alpha u_1^{t}\nabla z_\alpha\cdot\nabla
u_1.
$$
From the Cauchy-Schwarz inequality, we get
\begin{equation}
d_1^{2}\leq (t+1)^{2}\int_{\Omega}(z_\alpha
u_1)^{2}\int_{\Omega}u_1^{2t-2}(\nabla z_\alpha\cdot\nabla
u_1)^{2}.\end{equation} Hence, we have
\begin{equation}
\sum_{\alpha=1}^{n}\frac{1}{\|z_\alpha
u_1\|^{2}}\leq\frac{(t+1)^2\int_{\Omega}u_1^{2t-2}|\nabla
u_1|^{2}}{\left(\int_{\Omega}u_1^{t+1}\right)^{2}}=\frac{(t+1)^2}{2t-1}\lambda_1B(t).
\end{equation}

 For any fixed $j\geq 2$, we choose an orthonormal
 transformation
 $\widetilde{z}_\alpha=\sum\limits_{\beta=1}^{n}h_{\alpha\beta}z_\beta \ (1\leq\alpha\leq n)$
 with ${\rm det}(h_{\alpha\beta})=1$ such that
 $$\widetilde{a}_{\alpha j}=\int_\Omega \widetilde{z}_\alpha u_1u_j=\left
 \{ \aligned
 \sqrt{\sum_{\beta=1}^{j-1}a^2_{\beta j}}~, \quad
 & \alpha=1,\\
  0,\quad\quad\quad & \alpha\geq 1.\endaligned
\right. $$
 From the definition of $\widetilde{z}_{\alpha}$ and (2.4), we
 get
\begin{equation*}
\begin{aligned}
0&=-\frac 2{t+1}\int_{\Omega}u_1^{t +1}\Delta \widetilde{z}_{\alpha}
= 2\int_{\Omega}u_1^{t}\nabla \widetilde{z}_{\alpha}\cdot\nabla u_1\\
&= 2\sum_{p=1}^{\infty}d_{p}\int_{\Omega}u_p\nabla \widetilde{z}_{\alpha}\cdot\nabla u_1\\
&= 2\sum_{p=1}^{\infty}\sum_{\beta=1}^{n}d_{p}h_{\alpha\beta}b_{\beta p}\\
&=\sum_{p=1}^{\infty}\sum_{\beta=1}^{n}d_{p}h_{\alpha\beta}(\lambda_1-\lambda_p)a_{\beta
p}\\
&=\sum_{p=1}^{\infty}(\lambda_1-\lambda_p)d_{p}\widetilde{a}_{\alpha
p}.
\end{aligned}
\end{equation*}
Hence, from (2.13), we obtain
\begin{equation}
\sum_{p=2}^{\infty}(\lambda_p-\lambda_1)\beta_{p}\widetilde{a}_{\alpha
p}=0.
\end{equation}
Notice that \begin{equation}\aligned
1&=-2\int_{\Omega}\widetilde{z}_\alpha u_1
\nabla \widetilde{z}_\alpha \cdot \nabla u_1\\
&=-2\sum_{p=1}^{\infty}\widetilde{a}_{\alpha
p}h_{\alpha\beta}b_{\beta p}\\
&=-\sum_{p=1}^{\infty}\widetilde{a}_{\alpha
p}h_{\alpha\beta}(\lambda_1-\lambda_p)a_{\beta
p}\\
&=\sum_{p=2}^{\infty}(\lambda_p-\lambda_1)\widetilde{a}_{\alpha
p}^{2}.\endaligned\end{equation}
 Thus, from (2.20) and the Cauchy-Schwarz
inequality, we infer
\begin{equation*}
\biggl((\lambda_{j}-\lambda_1)\beta_{j}\widetilde{a}_{\alpha
j}\biggl)^2 \leq \left(\sum_{p=2, p\neq
j}^{\infty}(\lambda_p-\lambda_1)\beta_{p}^2\right)\left(\sum_{p=2,
p\neq j}^{\infty}(\lambda_p-\lambda_1)\widetilde{a}_{\alpha
p}^2\right).
\end{equation*}
Then, according to (2.21) and (2.14), we derive
\begin{equation*}
(\lambda_{j}-\lambda_1)^2\beta_{j}^2\widetilde{a}_{\alpha j}^2 \leq
\biggl(1-(\lambda_{j}-\lambda_1)\beta_{j}^2\biggl) \biggl(1-
(\lambda_{j}-\lambda_1)\widetilde{a}_{\alpha j}^2\biggl).
\end{equation*}
Hence, we have
$$
(\lambda_{j}-\lambda_1)\beta_{j}^2+
(\lambda_{j}-\lambda_1)\widetilde{a}_{\alpha j}^2\leq 1,$$ namely,
\begin{equation}
(\lambda_{j}-\lambda_1)\beta_{j}^2+
(\lambda_{j}-\lambda_1)\sum_{\beta=1}^{j-1}a_{\beta j}^2\leq 1.
\end{equation}

 From (2.8), (2.15), (2.17), (2.19) and (2.22), we have
$$\aligned
&\sum_{\alpha=1}^{n}\frac{\lambda_{k+1}-\lambda_{1}}{1+\sum\limits_{j=\alpha+1}^{k}
(\lambda_{k+1}-\lambda_{j})a_{\alpha j}^{2}}\\
\leq &\sum_{\alpha=1}^{n}\frac{1}{\|z_{\alpha}u_1\|^{2}}\\
 \leq &~
\frac{(t+1)^{2}}{2t-1}\lambda_{1}B(t)\\
=&~\dfrac{(t+1)^{2}}{2t-1}\dfrac{\lambda_{1}}{1-\dfrac{(t-1)^2}{2t-1}\lambda_1\sum\limits_{j=2}^{\infty}\beta_{
j}^2}\\
\leq
&~\dfrac{(t+1)^{2}}{2t-1}\dfrac{\lambda_{1}}{1-\dfrac{(t-1)^2}{2t-1}\dfrac{\lambda_1}{\lambda_{l+1}-\lambda_{1}}
\biggl(1+\sum\limits_{j=2}^{l}(\lambda_{l+1}-\lambda_{j})\beta_{
j}^2\biggl)}\\
 \leq
&~\dfrac{(t+1)^{2}}{2t-1}\dfrac{\lambda_{1}}{1-\dfrac{(t-1)^2}{2t-1}\dfrac{\lambda_1}{\lambda_{l+1}-\lambda_{1}}
\biggl(1+\sum\limits_{j=2}^{l}\dfrac{\lambda_{l+1}-\lambda_{j}}{\lambda_{j}-\lambda_{1}}\biggl[1-(\lambda_{j}-\lambda_{1})
\sum\limits_{\alpha=1}^{j-1}a_{\alpha j}^2\biggl]\biggl)}\\
=&~\dfrac{(t+1)^{2}}{2t-1}\dfrac{\lambda_{1}}{1-\dfrac{(t-1)^2}{2t-1}\dfrac{\lambda_1}{\sigma_{l}-\lambda_{1}}}.
\endaligned$$
Taking $t=\dfrac{2\sigma_l}{\sigma_l+\lambda_1},$ we obtain
(2.2).\endproof

In order to prove the theorem 1.1, we prepare the following lemmas.

\begin{lemma}
Let $\{\theta_i\}_{i=1}^{m+2}$ be an increasing real sequence and
let $\omega=(\omega_{jk})$ be a real $(m+1)\times (m+1)$-matrix.
Then the following equality holds:
$$\aligned &\sum_{i=1}^{m}\frac{\theta_{m+2}-\theta_{1}}{1+\sum\limits_{p=i+1}^{m+1}
(\theta_{m+2}-\theta_{p})\omega_{i
p}^{2}}-\sum_{i=1}^{m}(\theta_{i+1}-\theta_1)\\
=&\sum_{j=1}^{m}\sum_{i=1}^{m+1-j}
\frac{(\theta_{i+j+1}-\theta_{i+j})\biggl[1-\sum\limits_{p=i+1}^{i+j}(\theta_p-\theta_1)\omega_{i
p}^{2}\biggl]}{\biggl[1+\sum\limits_{p=i+1}^{i+j-1}(\theta_{i+j}-\theta_p)\omega_{i
p}^{2}\biggl]\biggl[1+\sum\limits_{p=i+1}^{i+j}(\theta_{i+j+1}-\theta_p)\omega_{i
p}^{2}\biggl]}.\endaligned$$
\end{lemma}

\proof \ \ For $1\leq j\leq m+1$, we define
$$\aligned
F_j=&\sum_{i=1}^{m+1-j}
\frac{(\theta_{m+2}-\theta_{i+j})\biggl[1-\sum\limits_{p=i+1}^{i+j}(\theta_p-\theta_1)\omega_{i
p}^{2}\biggl]-(\theta_{i+j}-\theta_{1})\sum\limits_{p=i+j+1}^{m+1}(\theta_{m+2}-\theta_p)\omega_{i
p}^{2}}{\biggl[1+\sum\limits_{p=i+1}^{i+j-1}(\theta_{i+j}-\theta_p)\omega_{i
p}^{2}\biggl]\biggl[1+\sum\limits_{p=i+1}^{m+1}(\theta_{m+2}-\theta_p)\omega_{i
p}^{2}\biggl]},\endaligned$$
$$\aligned
 G_j=&\sum_{i=1}^{m+1-j}
\frac{(\theta_{i+j+1}-\theta_{i+j})\biggl[1-\sum\limits_{p=i+1}^{i+j}(\theta_p-\theta_1)\omega_{i
p}^{2}\biggl]}{\biggl[1+\sum\limits_{p=i+1}^{i+j-1}(\theta_{i+j}-\theta_p)\omega_{i
p}^{2}\biggl]\biggl[1+\sum\limits_{p=i+1}^{i+j}(\theta_{i+j+1}-\theta_p)\omega_{i
p}^{2}\biggl]},\\
D_{ij}=&\biggl[1+\sum\limits_{p=i+1}^{i+j-1}(\theta_{i+j}-\theta_p)\omega_{i
p}^{2}\biggl]\biggl[1+\sum\limits_{p=i+1}^{i+j}(\theta_{i+j+1}-\theta_p)\omega_{i
p}^{2}\biggl]\biggl[1+\sum\limits_{p=i+1}^{m+1}(\theta_{m+2}-\theta_p)\omega_{i
p}^{2}\biggl].
\endaligned$$ Then, we have the following recursion
formula:
\begin{equation}
 F_j-G_j=F_{j+1}.\end{equation}
 In fact,$$\aligned
 &F_j-G_j\\
 =&\sum_{i=1}^{m+1-j}\frac{1}{D_{ij}}\biggl\{(\theta_{m+2}-\theta_{i+j})\biggl[1-\sum\limits_{p=i+1}^{i+j}(\theta_p-\theta_1)\omega_{i
p}^{2}\biggl]\biggl[1+\sum\limits_{p=i+1}^{i+j}(\theta_{i+j+1}-\theta_p)\omega_{i
p}^{2}\biggl]\\
& \ \ \
-(\theta_{i+j}-\theta_{1})\sum\limits_{p=i+j+1}^{m+1}(\theta_{m+2}-\theta_p)\omega_{i
p}^{2}\biggl[1+\sum\limits_{p=i+1}^{i+j}(\theta_{i+j+1}-\theta_p)\omega_{i
p}^{2}\biggl]\\
& \ \ \
-(\theta_{i+j+1}-\theta_{i+j})\biggl[1-\sum\limits_{p=i+1}^{i+j}(\theta_p-\theta_1)\omega_{i
p}^{2}\biggl]\biggl[1+\sum\limits_{p=i+1}^{m+1}(\theta_{m+2}-\theta_p)\omega_{i
p}^{2}\biggl]\biggl\}\\
=&\sum_{i=1}^{m+1-j}\frac{1}{D_{ij}}\biggl\{\biggl[1-\sum\limits_{p=i+1}^{i+j}(\theta_p-\theta_1)\omega_{i
p}^{2}\biggl]\biggl[(\theta_{m+2}-\theta_{i+j})
+(\theta_{m+2}-\theta_{i+j})\sum\limits_{p=i+1}^{i+j}(\theta_{i+j+1}-\theta_p)\omega_{i
p}^{2}\\
& \ \ \
 -(\theta_{i+j+1}-\theta_{i+j})
-(\theta_{i+j+1}-\theta_{i+j})\sum\limits_{p=i+1}^{m+1}(\theta_{m+2}-\theta_p)\omega_{i
p}^{2}\biggl]\\
& \ \ \
-(\theta_{i+j}-\theta_{1})\sum\limits_{p=i+j+1}^{m+1}(\theta_{m+2}-\theta_p)\omega_{i
p}^{2}\biggl[1+\sum\limits_{p=i+1}^{i+j}(\theta_{i+j+1}-\theta_p)\omega_{i
p}^{2}\biggl]\biggl\}\\
=&\sum_{i=1}^{m+1-j}\frac{1}{D_{ij}}\biggl\{\biggl[1-\sum\limits_{p=i+1}^{i+j}(\theta_p-\theta_1)\omega_{i
p}^{2}\biggl]\biggl[(\theta_{m+2}-\theta_{i+j+1})
+\sum_{p=i+1}^{i+j}(\theta_{m+2}-\theta_{i+j})(\theta_{i+j+1}-\theta_p)\omega_{i
p}^{2}\\
& \ \ \
-\sum_{p=i+1}^{i+j}(\theta_{i+j+1}-\theta_{i+j})(\theta_{m+2}-\theta_p)\omega_{i
p}^{2}-(\theta_{i+j+1}-\theta_{i+j})\sum_{p=i+j+1}^{m+1}(\theta_{m+2}-\theta_p)\omega_{i
p}^{2}\biggl]\\
& \ \ \
-(\theta_{i+j}-\theta_{1})\sum\limits_{p=i+j+1}^{m+1}(\theta_{m+2}-\theta_p)\omega_{i
p}^{2}\biggl[1+\sum\limits_{p=i+1}^{i+j}(\theta_{i+j+1}-\theta_p)\omega_{i
p}^{2}\biggl]\biggl\}\\
=&\sum_{i=1}^{m+1-j}\frac{1}{D_{ij}}\biggl\{\biggl[1-\sum\limits_{p=i+1}^{i+j}(\theta_p-\theta_1)\omega_{i
p}^{2}\biggl]\biggl[(\theta_{m+2}-\theta_{i+j+1})
+(\theta_{m+2}-\theta_{i+j+1})\sum_{p=i+1}^{i+j}(\theta_{i+j}-\theta_p)\omega_{i
p}^{2}\\
& \ \ \
-(\theta_{i+j+1}-\theta_{i+j})\sum_{p=i+j+1}^{m+1}(\theta_{m+2}-\theta_p)\omega_{i
p}^{2}\biggl]-(\theta_{i+j}-\theta_{1})\sum\limits_{p=i+j+1}^{m+1}(\theta_{m+2}-\theta_p)\omega_{i
p}^{2}\\
\endaligned$$
$$\aligned
& \ \ \
-(\theta_{i+j}-\theta_{1})\sum_{p=i+j+1}^{m+1}(\theta_{m+2}-\theta_p)\omega_{i
p}^{2}\sum\limits_{p=i+1}^{i+j}(\theta_{i+j+1}-\theta_p)\omega_{i
p}^{2}\biggl\}\\
&=\sum_{i=1}^{m+1-j}\frac{1}{D_{ij}}\biggl\{(\theta_{m+2}-\theta_{i+j+1})\biggl[1-\sum_{p=i+1}^{i+j}(\theta_p-\theta_1)\omega_{i
p}^{2}\biggl]\biggl[
1+\sum_{p=i+1}^{i+j-1}(\theta_{i+j}-\theta_p)\omega_{i
p}^{2}\biggl]\\
& \ \ \
-(\theta_{i+j+1}-\theta_{i+j})\sum_{p=i+j+1}^{m+1}(\theta_{m+2}-\theta_p)\omega_{i
p}^{2}\biggl[1-\sum_{p=i+1}^{i+j}(\theta_p-\theta_1)\omega_{i
p}^{2}\biggl]\\
 & \ \ \
-(\theta_{i+j}-\theta_{1})\sum\limits_{p=i+j+1}^{m+1}(\theta_{m+2}-\theta_p)\omega_{i
p}^{2}\\
& \ \ \
-(\theta_{i+j}-\theta_{1})\sum_{p=i+j+1}^{m+1}(\theta_{m+2}-\theta_p)\omega_{i
p}^{2}\sum\limits_{p=i+1}^{i+j}(\theta_{i+j+1}-\theta_p)\omega_{i
p}^{2}\biggl\}\\
&=\sum_{i=1}^{m+1-j}\frac{1}{D_{ij}}\biggl\{(\theta_{m+2}-\theta_{i+j+1})\biggl[1-\sum_{p=i+1}^{i+j}(\theta_p-\theta_1)\omega_{i
p}^{2}\biggl]\biggl[
1+\sum_{p=i+1}^{i+j-1}(\theta_{i+j}-\theta_p)\omega_{i
p}^{2}\biggl]\\
& \ \ \
-(\theta_{i+j+1}-\theta_{1})\sum_{p=i+j+1}^{m+1}(\theta_{m+2}-\theta_p)\omega_{i
p}^{2}\biggl[1+\sum_{p=i+1}^{i+j}(\theta_{i+j}-\theta_p)\omega_{i
p}^{2}\biggl]\biggl\}\\
&=\sum_{i=1}^{m+1-j}\frac{\biggl[1+\sum\limits_{p=i+1}^{i+j-1}(\theta_{i+j}-\theta_p)\omega_{i
p}^{2}\biggl]}{D_{ij}}\biggl\{(\theta_{m+2}-\theta_{i+j+1})\biggl[1-\sum_{p=i+1}^{i+j}(\theta_p-\theta_1)\omega_{i
p}^{2}\biggl]\\
& \ \ \
-(\theta_{i+j+1}-\theta_{1})\sum_{p=i+j+1}^{m+1}(\theta_{m+2}-\theta_p)\omega_{i
p}^{2}\biggl\}\\
&=\sum_{i=1}^{m-j}\frac{(\theta_{m+2}-\theta_{i+j+1})\biggl[1-\sum\limits_{p=i+1}^{i+j+1}(\theta_p-\theta_1)\omega_{i
p}^{2}\biggl]
-(\theta_{i+j+1}-\theta_{1})\sum\limits_{p=i+j+2}^{m+1}(\theta_{m+2}-\theta_p)\omega_{i
p}^{2}}{\biggl[1+\sum\limits_{p=i+1}^{i+j}(\theta_{i+j+1}-\theta_p)\omega_{i
p}^{2}\biggl]\biggl[1+\sum\limits_{p=i+1}^{m+1}(\theta_{m+2}-\theta_p)\omega_{i
p}^{2}\biggl]}\\
&=F_{j+1}.
\endaligned
$$
 Therefore, we have
$$\aligned
&\sum_{i=1}^{m}\frac{\theta_{m+2}-\theta_{1}}{1+\sum\limits_{p=i+1}^{m+1}
(\theta_{m+2}-\theta_{p})\omega_{i
p}^{2}}-\sum_{i=1}^{m}(\theta_{i+1}-\theta_1)\\
=&\sum_{i=1}^{m}\frac{(\theta_{m+2}-\theta_{i+1})\biggl[1-(\theta_{i+1}-\theta_1)\omega_{i
i+1}^{2}\biggl]-(\theta_{i+1}-\theta_1)\sum\limits_{p=i+2}^{m+1}
(\theta_{m+2}-\theta_{p})\omega_{ip}^{2}}{1+\sum\limits_{p=i+1}^{m+1}
(\theta_{m+2}-\theta_{p})\omega_{ip}^{2}}\\
=&~F_1=G_1+F_2=G_1+G_2+F_3=\cdots=\sum_{j=1}^{m}G_j+F_{m+1}.\endaligned$$
Since $F_{m+1}=0,$ we complete the proof of the lemma 2.1.
\endproof

\begin{lemma} Let $\{\theta_i\}_{i=1}^{m+2}$ be an increasing real sequence and let
$\omega=(\omega_{jk})$ be a real $(m+1)\times (m+1)$-matrix. Then
the following equality holds:
$$\aligned &\sum_{j=1}^{m}\sum_{i=1}^{m+1-j}
\frac{(\theta_{i+j+1}-\theta_{i+j})(\theta_{2}-\theta_{1})^{2}}
{(\theta_{i+j+1}-\theta_{1})(\theta_{i+j}-\lambda_{1})}
\biggl[1-\sum\limits_{p=i+1}^{i+j}(\theta_p-\theta_1)\omega_{i
p}^{2}\biggl]\\
=&\sum_{j=2}^{m+1}
\frac{(\theta_{j+1}-\theta_{j})(\theta_{2}-\theta_{1})^{2}}
{(\theta_{j+1}-\theta_{1})(\theta_{j}-\lambda_{1})}
\sum_{i=1}^{j-1}\biggl[1-\sum\limits_{p=i+1}^{j}(\theta_p-\theta_1)\omega_{i
p}^{2}\biggl].\endaligned$$
\end{lemma}

\begin{proof}  In fact, we have
$$\aligned &\sum_{j=1}^{m}\sum_{i=1}^{m+1-j}
\frac{(\theta_{i+j+1}-\theta_{i+j})(\theta_{2}-\theta_{1})^{2}}
{(\theta_{i+j+1}-\theta_{1})(\theta_{i+j}-\lambda_{1})}
\biggl[1-\sum\limits_{p=i+1}^{i+j}(\theta_p-\theta_1)\omega_{i
p}^{2}\biggl]\\
=& ~\frac{(\theta_{m+2}-\theta_{m+1})(\theta_{2}-\theta_{1})^{2}}
{(\theta_{m+2}-\theta_{1})(\theta_{m+1}-\lambda_{1})}\biggl[1-\sum_{p=2}^{m+1}(\theta_p-\theta_1)\omega_{1
p}^{2}\biggl]\\
 &+\sum_{i=1}^{2}
\frac{(\theta_{i+m}-\theta_{i+m-1})(\theta_{2}-\theta_{1})^{2}}
{(\theta_{i+m}-\theta_{1})(\theta_{i+m-1}-\lambda_{1})}
\biggl[1-\sum_{p=i+1}^{i+m-1}(\theta_p-\theta_1)\omega_{i
p}^{2}\biggl]\\
&+\cdots+\sum_{i=1}^{m-1}
\frac{(\theta_{i+3}-\theta_{i+2})(\theta_{2}-\theta_{1})^{2}}
{(\theta_{i+3}-\theta_{1})(\theta_{i+2}-\lambda_{1})}\biggl[1-\sum_{p=i+1}^{i+2}(\theta_p-\theta_1)\omega_{i
p}^{2}\biggl]\\
&+\sum_{i=1}^{m}
\frac{(\theta_{i+2}-\theta_{i+1})(\theta_{2}-\theta_{1})^{2}}
{(\theta_{i+2}-\theta_{1})(\theta_{i+1}-\lambda_{1})}
\biggl[1-(\theta_{i+1}-\theta_1)\omega_{i i+1}^{2}\biggl]\\
=&~\frac{(\theta_{m+2}-\theta_{m+1})(\theta_{2}-\theta_{1})^{2}}
{(\theta_{m+2}-\theta_{1})(\theta_{m+1}-\lambda_{1})}\sum_{i=1}^{m}\biggl[1-\sum_{p=i+1}^{m+1}(\theta_p-\theta_1)\omega_{i
p}^{2}\biggl]\\
&+~\frac{(\theta_{m+1}-\theta_{m})(\theta_{2}-\theta_{1})^{2}}
{(\theta_{m+1}-\theta_{1})(\theta_{m}-\lambda_{1})}\sum_{i=1}^{m-1}\biggl[1-\sum_{p=i+1}^{m}(\theta_p-\theta_1)\omega_{i
p}^{2}\biggl]\\
&+\cdots+\frac{(\theta_{3}-\theta_{2})(\theta_{2}-\theta_{1})^{2}}
{(\theta_{3}-\theta_{1})(\theta_{2}-\lambda_{1})}\biggl[1-(\theta_2-\theta_1)\omega_{1
2}^{2}\biggl]\\
 =&\sum_{j=2}^{m+1}
\frac{(\theta_{j+1}-\theta_{j})(\theta_{2}-\theta_{1})^{2}}
{(\theta_{j+1}-\theta_{1})(\theta_{j}-\lambda_{1})}
\sum_{i=1}^{j-1}\biggl[1-\sum\limits_{p=i+1}^{j}(\theta_p-\theta_1)\omega_{i
p}^{2}\biggl].\endaligned
$$
\end{proof}

\begin{lemma} Let $\{\theta_i\}_{i=1}^{m+2}$ be an increasing real sequence and let
$\omega=(\omega_{jk})$ be a real $(m+1)\times (m+1)$-matrix. Then,
for $1\leq i\leq m+2$, the following equality holds:
$$\aligned
&\sum\limits_{j=2}^{i-1}\dfrac{\theta_{i}-\theta_j}{\theta_{j}-\theta_1}
\biggl[1-(\theta_{j}-\theta_1)\sum\limits_{k=1}^{j-1}\omega_{k
j}^2\biggl]\\
 =&
\sum_{j=2}^{i-1}\frac{(\theta_{i}-\theta_1)(\theta_{j+1}-\theta_j)}{(\theta_{j+1}-\theta_1)(\theta_{j}-\theta_1)}
\sum_{k=1}^{j-1}\biggl[1-\sum_{p=k+1}^{j}(\theta_{p}-\theta_1)\omega_{k
p}^2\biggl].\endaligned$$
\end{lemma}

\proof \ \ For $1\leq i \leq m+2$, we have
$$\aligned
&\sum\limits_{j=2}^{i-1}\dfrac{\theta_{i}-\theta_j}{\theta_{j}-\theta_1}
\biggl[1-(\theta_{j}-\theta_1)\sum\limits_{k=1}^{j-1}\omega_{k
j}^2\biggl]\\
=&\sum\limits_{j=2}^{i-1}\dfrac{\theta_{i}-\theta_j}{\theta_{j}-\theta_1}
\sum_{k=1}^{j-1}\biggl[1-\sum\limits_{p=k+1}^{j}(\theta_{p}-\theta_1)\omega_{k
p}^2\biggl]
-\sum\limits_{j=3}^{i-1}\dfrac{\theta_{i}-\theta_j}{\theta_{j}-\theta_1}
\sum_{k=1}^{j-2}\biggl[1-\sum\limits_{p=k+1}^{j-1}(\theta_{p}-\theta_1)\omega_{k
p}^2\biggl]\\
=&\sum\limits_{j=2}^{i-1}\dfrac{\theta_{i}-\theta_j}{\theta_{j}-\theta_1}
\sum_{k=1}^{j-1}\biggl[1-\sum\limits_{p=k+1}^{j}(\theta_{p}-\theta_1)\omega_{k
p}^2\biggl]
-\sum\limits_{j=2}^{i-2}\dfrac{\theta_{i}-\theta_{j+1}}{\theta_{j+1}-\theta_1}
\sum_{k=1}^{j-1}\biggl[1-\sum\limits_{p=k+1}^{j}(\theta_{p}-\theta_1)\omega_{k
p}^2\biggl]\\
=&\frac{\theta_{i}-\theta_{i-1}}{\theta_{i-1}-\theta_1}\sum_{k=1}^{i-2}\biggl[1-\sum_{p=k+1}^{i-1}(\theta_{p}-\theta_1)\omega_{k
p}^2\biggl]+\sum_{j=2}^{i-2}\frac{(\theta_{i}-\theta_1)(\theta_{j+1}-\theta_j)}{(\theta_{j+1}-\theta_1)(\theta_{j}-\theta_1)}
\sum_{k=1}^{j-1}\biggl[1-\sum_{p=k+1}^{j}(\theta_{p}-\theta_1)\omega_{k
p}^2\biggl]\\
 =&
\sum_{j=2}^{i-1}\frac{(\theta_{i}-\theta_1)(\theta_{j+1}-\theta_j)}{(\theta_{j+1}-\theta_1)(\theta_{j}-\theta_1)}
\sum_{k=1}^{j-1}\biggl[1-\sum_{p=k+1}^{j}(\theta_{p}-\theta_1)\omega_{k
p}^2\biggl].\endaligned$$\endproof

\vskip 2mm \noindent {\it Proof of Theorem 1.1.} If
$\dfrac{\lambda_2}{\lambda_1}<2-\dfrac{\lambda_1}{\lambda_i}$, we
know that the theorem 1.1 is proved. Hence, we assume
$\dfrac{\lambda_2}{\lambda_1}\geq 2-\dfrac{\lambda_1}{\lambda_i}$.
Taking $k=n+1, \ l=i-1$ in (2.2), we have
\begin{equation}
\sum_{\alpha=1}^{n}(\lambda_{\alpha+1}-\lambda_1)+B\leq
3\lambda_1+\dfrac{\lambda_1^{2}}{\lambda_i}+C,
\end{equation}
where
$$\aligned
B&=\sum_{\alpha=1}^{n}\frac{\lambda_{n+2}-\lambda_{1}}{1+\sum\limits_{j=\alpha+1}^{n+1}
(\lambda_{n+2}-\lambda_{j})a_{\alpha
j}^{2}}-\sum_{\alpha=1}^{n}(\lambda_{\alpha+1}-\lambda_1),\\
C&=\frac{(\lambda_{i}-\lambda_{1})\sum\limits_{j=2}^{i-1}\dfrac{\lambda_{i}-\lambda_j}{\lambda_{j}-\lambda_1}
\biggl[1-(\lambda_{j}-\lambda_1)\sum\limits_{\alpha=1}^{j-1}a_{\alpha
j}^2\biggl]}{\dfrac{\lambda_i}{\lambda_1}\biggl\{\dfrac{\lambda_i}{\lambda_1}
+\sum\limits_{j=2}^{i-1}\dfrac{\lambda_{i}-\lambda_j}{\lambda_{j}-\lambda_1}
\biggl[1-(\lambda_{j}-\lambda_1)\sum\limits_{\alpha=1}^{j-1}a_{\alpha
j}^2\biggl]\biggl\}}.
\endaligned$$
It follows from the lemma 2.1 that
\begin{equation}
B=\sum_{\beta=1}^{n}\sum_{\alpha=1}^{n+1-\beta}
\frac{(\lambda_{\alpha+\beta+1}-\lambda_{\alpha+\beta})\biggl[1-\sum\limits_{p=\alpha+1}^{\alpha+\beta}(\lambda_p-\lambda_1)a_{\alpha
p}^{2}\biggl]}{\biggl[1+\sum\limits_{p=\alpha+1}^{\alpha+\beta-1}(\lambda_{\alpha+\beta}-\lambda_p)a_{\alpha
p}^{2}\biggl]\cdot\biggl[1+\sum\limits_{p=\alpha+1}^{\alpha+\beta}(\lambda_{\alpha+\beta+1}-\lambda_p)a_{\alpha
p}^{2}\biggl]}.\end{equation}

For any positive integer $\gamma$, from (2.7), we can get
$$\sum_{p=\alpha+1}^{\gamma}(\lambda_p-\lambda_1)a_{\alpha
p}^{2}\leq 1.$$ Then, we have the following recursion formula:
$$\aligned
&1+\sum\limits_{p=\alpha+1}^{\gamma}(\lambda_{\gamma+1}-\lambda_p)a_{\alpha
p}^{2}\\
=~&
1+\frac{\lambda_{\gamma+1}-\lambda_{\gamma}}{\lambda_{\gamma}-\lambda_1}
\sum_{p=\alpha+1}^{\gamma}(\lambda_p-\lambda_1)a_{\alpha
p}^{2}-\frac{\lambda_{\gamma+1}-\lambda_{\gamma}}{\lambda_{\gamma}-\lambda_1}
\sum_{p=\alpha+1}^{\gamma-1}(\lambda_p-\lambda_1)a_{\alpha p}^{2}
+\sum\limits_{p=\alpha+1}^{\gamma-1}(\lambda_{\gamma+1}-\lambda_p)a_{\alpha
p}^{2}\\
\leq ~ &
1+\frac{\lambda_{\gamma+1}-\lambda_{\gamma}}{\lambda_{\gamma}-\lambda_1}
-\frac{\lambda_{\gamma+1}-\lambda_{\gamma}}{\lambda_{\gamma}-\lambda_1}
\sum_{p=\alpha+1}^{\gamma-1}(\lambda_p-\lambda_1)a_{\alpha p}^{2}
+\sum\limits_{p=\alpha+1}^{\gamma-1}(\lambda_{\gamma+1}-\lambda_p)a_{\alpha
p}^{2}\\
=~&
\frac{\lambda_{\gamma+1}-\lambda_{1}}{\lambda_{\gamma}-\lambda_1}\biggl[1+
\sum_{p=\alpha+1}^{\gamma-1}(\lambda_{\gamma}-\lambda_{p})a_{\alpha
p}^{2}\biggl].
\endaligned$$
Therefore, we can obtain
$$\aligned
&1+\sum\limits_{p=\alpha+1}^{\alpha+\beta}(\lambda_{\alpha+\beta+1}-\lambda_p)a_{\alpha
p}^{2}\\
\leq~&\frac{\lambda_{\alpha+\beta+1}-\lambda_{1}}{\lambda_{\alpha+\beta}-\lambda_1}
\cdot\frac{\lambda_{\alpha+\beta}-\lambda_{1}}{\lambda_{\alpha+\beta-1}-\lambda_1}\cdots
\frac{\lambda_{\alpha+2}-\lambda_{1}}{\lambda_{\alpha+1}-\lambda_1}
=\frac{\lambda_{\alpha+\beta+1}-\lambda_{1}}{\lambda_{\alpha+1}-\lambda_1}.\endaligned$$
Taking analogous arguments as the above inequality, we can obtain
$$
1+\sum\limits_{p=\alpha+1}^{\alpha+\beta-1}(\lambda_{\alpha+\beta}-\lambda_p)a_{\alpha
p}^{2}\leq\frac{\lambda_{\alpha+\beta}-\lambda_{1}}{\lambda_{\alpha+1}-\lambda_1}.$$

From (2.25), the above inequalities and the lemma 2.2, we have
$$\aligned
B\geq &\sum_{\beta=1}^{n}\sum_{\alpha=1}^{n+1-\beta}
\frac{(\lambda_{\alpha+\beta+1}-\lambda_{\alpha+\beta})(\lambda_{\alpha+1}-\lambda_{1})^{2}}
{(\lambda_{\alpha+\beta+1}-\lambda_{1})(\lambda_{\alpha+\beta}-\lambda_{1})}
\biggl[1-\sum\limits_{p=\alpha+1}^{\alpha+\beta}(\lambda_p-\lambda_1)a_{\alpha
p}^{2}\biggl]\\
\geq &\sum_{\beta=1}^{n}\sum_{\alpha=1}^{n+1-\beta}
\frac{(\lambda_{\alpha+\beta+1}-\lambda_{\alpha+\beta})(\lambda_{2}-\lambda_{1})^{2}}
{(\lambda_{\alpha+\beta+1}-\lambda_{1})(\lambda_{\alpha+\beta}-\lambda_{1})}
\biggl[1-\sum\limits_{p=\alpha+1}^{\alpha+\beta}(\lambda_p-\lambda_1)a_{\alpha
p}^{2}\biggl]\\
=&\sum_{\beta=2}^{n+1}
\frac{(\lambda_{\beta+1}-\lambda_{\beta})(\lambda_{2}-\lambda_{1})^{2}}
{(\lambda_{\beta+1}-\lambda_{1})(\lambda_{\beta}-\lambda_{1})}
\sum_{\alpha=1}^{\beta-1}\biggl[1-\sum\limits_{p=\alpha+1}^{\beta}(\lambda_p-\lambda_1)a_{\alpha
p}^{2}\biggl].
\endaligned
$$

 On the other hand, for $\dfrac{\lambda_2}{\lambda_1}\geq 2-\dfrac{\lambda_1}{\lambda_i}$, we have
 $$\frac{\lambda_i}{\lambda_1}\geq \frac{\lambda_i-\lambda_1}{\lambda_2-\lambda_1}.$$
 Hence, from the lemma 2.3, we deduce
$$\aligned
C&\leq\frac{(\lambda_{i}-\lambda_{1})}{\biggl(\dfrac{\lambda_i}{\lambda_1}\biggl)^{2}}
\sum\limits_{j=2}^{i-1}\dfrac{\lambda_{i}-\lambda_j}{\lambda_{j}-\lambda_1}
\biggl[1-(\lambda_{j}-\lambda_1)\sum\limits_{\alpha=1}^{j-1}a_{\alpha
j}^2\biggl]\\
&=
\frac{(\lambda_{i}-\lambda_{1})}{\biggl(\dfrac{\lambda_i}{\lambda_1}\biggl)^{2}}
\sum_{j=2}^{i-1}\frac{(\lambda_{i}-\lambda_1)(\lambda_{j+1}-\lambda_j)}{(\lambda_{j+1}-\lambda_1)(\lambda_{j}-\lambda_1)}
\sum_{\alpha=1}^{j-1}\biggl[1-\sum_{p=\alpha+1}^{j}(\lambda_{p}-\lambda_1)a_{\alpha
p}^2\biggl]\\
&\leq
\sum_{j=2}^{i-1}\frac{(\lambda_{2}-\lambda_1)^{2}(\lambda_{j+1}-\lambda_j)}{(\lambda_{j+1}-\lambda_1)(\lambda_{j}-\lambda_1)}
\sum_{\alpha=1}^{j-1}\biggl[1-\sum_{p=\alpha+1}^{j}(\lambda_{p}-\lambda_1)a_{\alpha
p}^2\biggl]\\
&\leq
\sum_{j=2}^{n+1}\frac{(\lambda_{2}-\lambda_1)^{2}(\lambda_{j+1}-\lambda_j)}{(\lambda_{j+1}-\lambda_1)(\lambda_{j}-\lambda_1)}
\sum_{\alpha=1}^{j-1}\biggl[1-\sum_{p=\alpha+1}^{j}(\lambda_{p}-\lambda_1)a_{\alpha
p}^2\biggl]\leq B.\\
\endaligned$$
Finally, we obtain,  from (2.24) and the above inequalities,
$$
\sum_{\alpha=1}^{n}(\lambda_{\alpha+1}-\lambda_1)\leq 3\lambda_1+\frac{\lambda_1^2}{\lambda_i}.
$$
This completes the proof of the theorem 1.1.
\endproof

%-----------------------------------------------------------------

\providecommand{\bysame}{\leavevmode\hbox
to3em{\hrulefill}\thinspace}

\end{document}